\documentclass[12pt]{amsart}
\usepackage[english]{babel}
\usepackage[pdftex]{graphicx}
\usepackage{wrapfig}
\usepackage{pdfsync}
\usepackage{epsfig}
\usepackage{epstopdf}
\usepackage{latexsym,amssymb,amsfonts,amsmath, amscd, euscript, MnSymbol}
\usepackage{mathrsfs}
\usepackage{enumerate}
\usepackage{verbatim}

\usepackage{geometry}                		
\usepackage{graphicx}				
\usepackage{amssymb}

\def\Power #1 { \powerset(#1) }
\def\Bidom #1 { {\mathfrak P} (#1) }

\newtheorem{definition}{{\bf Definition}}[section]
\newtheorem{theo}[definition]{{\bf Theorem}}

\newtheorem{cor}[definition]{{\bf Corollary}}

\newtheorem{proposition}[definition]{\noindent {\bf Proposition}}
\newtheorem{lem}[definition]{\noindent {\bf Lemma}}

\def\proofref #1 {{\noindent  {\bf Proof} (#1).}\ }

\def\endproof{\hfill {\kern 6pt\penalty 500
\raise -0pt\hbox{\vrule \vbox to5pt {\hrule width 5pt
\vfill\hrule}\vrule}}}
\def\centerpicture #1 by #2 (#3){\leavevmode
        \vbox to #2{
        \hrule width #1 height 0pt depth 0pt
        \vfill
        \special{pictfile #3}}}
\title[Absolute retracts of graphs]{Absolute retracts of reflexive oriented graphs: the role of the MacNeille completion\dagger}
\author[H-J. Bandelt]{Hans-J\"urgen Bandelt
}

\address{FB Mathematik, Universit\"at  Hamburg, Bundesstr. 55, D-20146 Hamburg, Germany.}
\email{bandelt@math.uni-hamburg.de}

\author[M. Pouzet]{Maurice Pouzet} \address{Univ Lyon, Universit\'e Claude Bernard Lyon 1, CNRS UMR 5208, Institut Camille Jordan, 43 blvd. du 11 novembre 1918, F-69622 Villeurbanne cedex, France and Mathematics \& Statistics Department, University of Calgary, Calgary, Alberta, Canada T2N 1N4}
 \email{pouzet@univ-lyon1.fr }

 \author[F. Sa\"{\i}dane]{Faouzi Sa\"{\i}dane} \address{D\'epartement de Math\'ematiques, Facult\'e des Sciences, 2100 Gafsa, Tunisie}
 \email{faouzisaidane22@yahoo.fr}

\keywords{ absolute retracts, oriented graphs,  MacNeille completion, dual quantal, generalized metric}
\subjclass[2020]{06A07,06A12,06D22,08B30, 68R10}

\date{\today}

\begin{document}

\dedicatory {\dagger \; Dedicated to the memory of Maurice Nivat}

\begin{abstract} We characterize the absolute retracts in the category of reflexive oriented graphs, that is, antisymmetric reflexive graphs, where morphisms between objects preserve arcs (which may be sent to loops). Here we show,  by correcting a much earlier attempt at a proof,  that a reflexive oriented graph is an absolute retract if and only if it is indeed a retract of some (direct) product of reflexive oriented zigzags (which are concatenations of reflexive oriented paths). Absolute retracts are therefore necessarily acyclic. In contrast to other categories of graphs and ordered sets, not every acyclic oriented graph can be embedded isometrically into some absolute retract. Embedding involves isometry with respect to the zig-zag  distances forming a particular "dual quantal", which is a complete lattice of certain sets of words over the alphabet $\{+, -\}$, endowed with an additional monoid operation (viz., compound concatenation of sets of words) and an involution (interchanging $+$ and $- $ and then mirroring words). As reflexive oriented zigzags have MacNeille-closed distances, so do their products and retracts. So, the category of reflexive oriented graphs and its full subcategory of reflexive acyclic graphs do not have enough injectives, as the injective objects coincide with the absolute retracts.
\end{abstract}

\maketitle

\section{Introduction}

This article is about retracts of reflexive oriented graphs, that is, oriented graphs with loops. "Oriented" is a shorthand for "antisymmetric directed". The reason for requiring a loop at every vertex is simply to permit homomorphisms sending two distinct vertices linked by an arc to a single vertex. For brevity we will henceforth  assume that \emph{all directed graphs in this paper are understood as having a loop at each vertex}.
The scenario for retracts in oriented graphs bears some resemblance with the undirected case of reflexive graphs, which was studied many years ago \cite{NoRi}, \cite{quilliot1, quilliot2}. The latter structures, simply referred to as graphs here  are pairs $G:= (V, \mathcal E)$, where $V$ is the set of \emph{vertices}, $\mathcal E$ is the set of \emph{edges}; each edge is an unordered pair $\{x,y\}$ of vertices, which is a  \emph{loop} if $x=y$.  Given a  graph $G$, we denote by $V(G)$ its set of vertices and $\mathcal E(G)$ its set of edges.

A \emph{homomorphism} $f$ from $G$ to $G'$ is a map $f: V(G) \rightarrow V(G')$ such that the image $\{f(x),f(y)\}$ of any  edge  $\{x,y\}$ is an edge and thus may be  a loop.
A graph $G$ is a \emph{retract} of $G'$ if there are homomorphisms $f$ from $G$ to $G'$ and $g$ from  $G'$ to $G$ such that the composition $g\circ f$ is the identity map on $G$. The map $f$ is called a \emph{coretraction} and the map $g$ a \emph{retraction}.
A coretraction $f
$ is necessarily an \emph{embedding}, that is,  for every pair of distinct vertices, the pair  $\{x,y\}$ is an edge exactly when  its image consists of distinct vertices forming  an edge. Graphs for which embeddings are coretractions are the so-called \emph{central graphs} (graphs containing a vertex linked to every other vertex).  Since not every graph is central, not every embedding is a coretraction in general. In fact, coretractions are necessarily isometric embeddings with respect to the \emph{graphic distance} (defined  as the length $d_G(x,y)$ of the shortest path joining vertices $x$ and $y$ in $G$ if there is any and $\infty$ if is are none), that is,  $d_{G'}(f(x),f(y))=d_G(x,y)$ for every pair  of vertices $x,y$ of $G$. Graphs for which isometric embeddings are coretractions are called \emph{absolute retracts}. These graphs have been characterized  as retracts of product of paths, see  Quilliot \cite{quilliot1, quilliot2},  Nowakowski  and Rival \cite {NoRi}.

The study of retracts of directed graphs  started with the work of Quilliot \cite{quilliot1}.
The oriented versions of undirected paths (with loops) are "zigzags". Coding
forward arcs by $"+"$ and backward arcs by $"-"$ we can associate to every zigzag some word over the alphabet $\Lambda:= \{+,-\}$.  For an example, the zigzag from $a$ to $b$ shown in Figure \ref{orientzigzag}  receives the
code $+++--+--$ and the reverse zigzag from $b$ to $a$ is coded by
$++-++---$.
The distance $d_G(x,y)$ between two vertices $x$ and $y$ of a directed graph $G:= (V,\mathcal E)$ is no longer an integer (or $+\infty$), but a certain set of words over the alphabet $\Lambda:= \{+, -\}$. Each word  $\alpha:= a_0\cdots a_i\cdots a_{n-1}$ codes a  sequence $z_0:=x, \dots, z_i, \dots , z_n:= y$ of vertices such that $(z_i,z_{i+1})\in \mathcal E$ if $a_i=+$ and $(z_{i+1},z_{i})\in \mathcal E$ if $a_i=-$. Such an oriented graph is called an \emph{oriented zigzag} (or just \emph{zigzag}, for short). In the more general class of directed graphs, zigzags may have a forward and a backward edge for a pair of distinct vertices; see Figure \ref{directzigzag}.
 We also call \emph{zigzags} their images and \emph{zigzag distance} the set $d_G(x,y)$.
Since the graphs we consider are reflexive, we may repeat any vertex  in a zigzag path, hence we may insert letters in its code.
Thus, zigzag distances are closed under the formation of superwords. The zigzag distance from a vertex to itself is defined as the set of all words (including the empty word). With this notion, one may define isometries and then absolute retracts.  A complete description of absolute retracts was given in \cite{kabil-pouzet}.

\begin{figure}[h]
\centering
\includegraphics[width=0.8\textwidth]{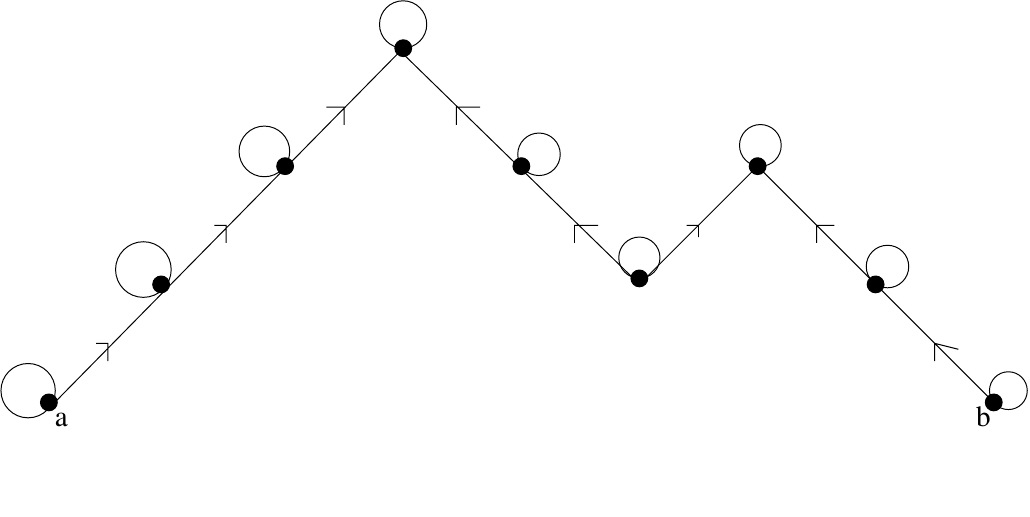}
\caption{A reflexive oriented zigzag}
\label{orientzigzag}
\end{figure}

\begin{figure}[h]
\centering
\includegraphics[width=0.8\textwidth]{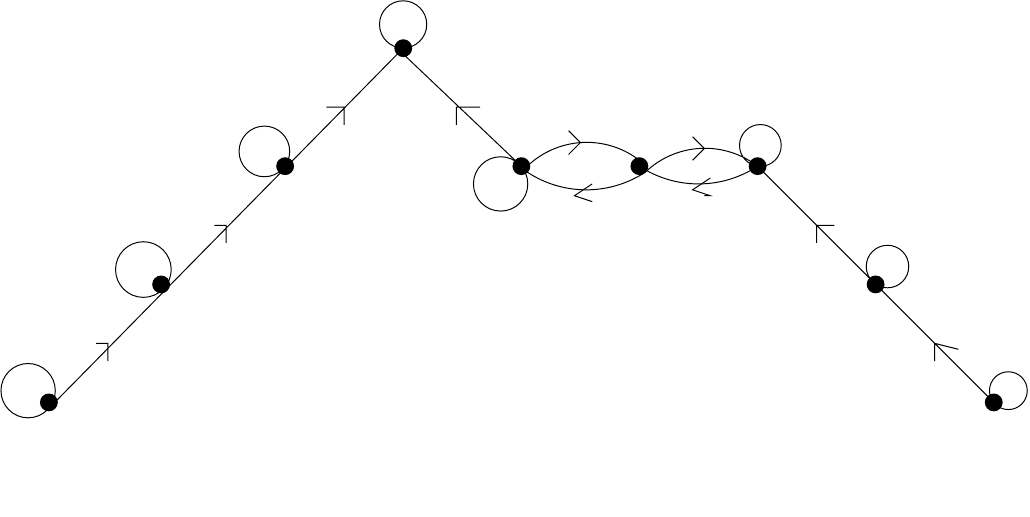}
\caption{A reflexive directed zigzag}
\label{directzigzag}
\end{figure}

\begin{figure}[h]
\centering
\includegraphics[width=0.8\textwidth]{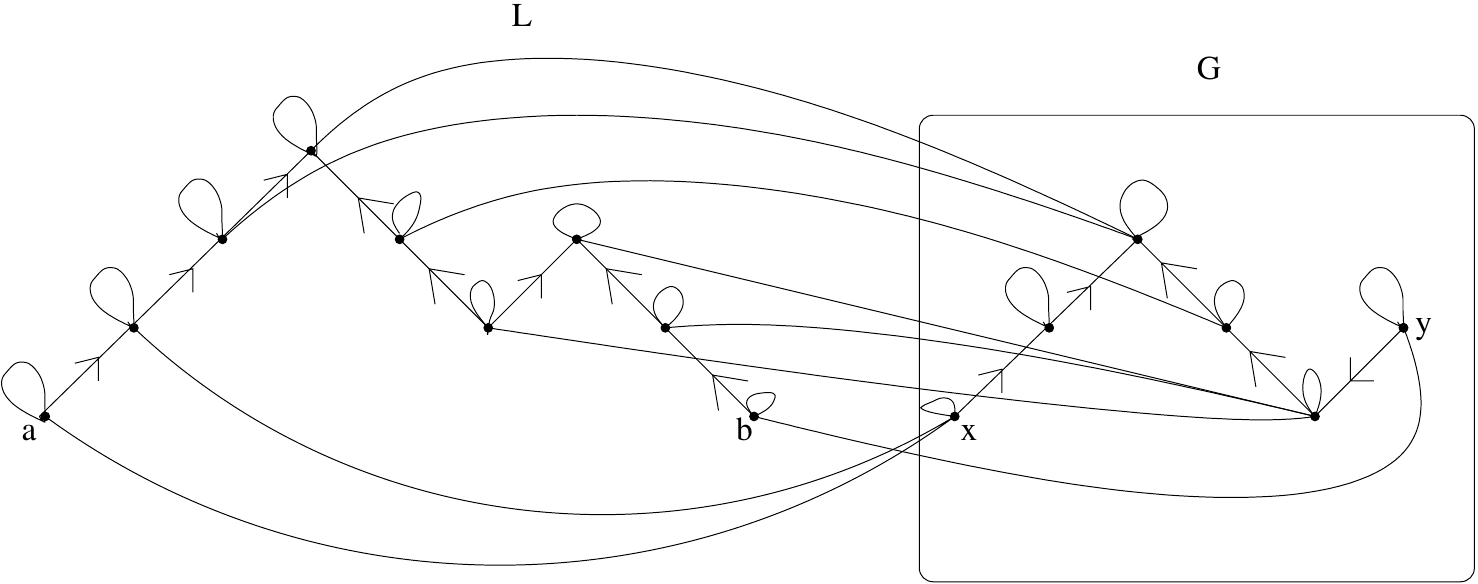}
\caption{A morphism of an oriented zigzag into a directed graph}
\label{figure-B-P-S-2019}
\end{figure}

\subsection{Oriented graphs and zigzags}
Consider now oriented graphs (with loops): i.e., directed graphs with loops in which every pair of vertices is linked by at most
one arc.
Basic instances are transitive oriented graphs, which are just ordered sets, alias posets.  In this case, graph homomorphisms become order preserving maps. Again, coretractions are embeddings, and ordered sets  for which embeddings are coretractions are complete lattices (Banaschewski and Bruns \cite {banaschewski-bruns}). For non-complete lattices, we may adapt the notion of graphic distance. Observe that a path has two orientations turning it into an ordered set,  called \emph{fence}. In an ordered set, say $P$, define the distance $d_P(x,y)$ from $x$ to $y$  as $\infty$  if no fence connects $x$ to $y$, otherwise among the fences $x:= a_0<a_1>a_2\dots$ going from $x$ to $y$ select the length, say $n$, of the shortest one; do the same for the fences $x:= a_0>a_1<a_2\dots$, let $m$ be the length of the shortest and set $d_P(x,y) = (n,m)$. Ordered sets  for which isometric embeddings are coretractions  are called absolute retracts; they have been  characterized as retracts of product of fences by Nevermann and  Rival, \cite{nevermann-rival} (see also \cite{JaMiPo}).

\subsection{Organization of the article} Our main result (Theorem \ref{theo-antisym})  asserts that  among oriented graphs the absolute retracts with respect to isometries are retracts of products of zigzags. Here a map $f$ from an oriented graph $G$ into an oriented graph $G'$ qualifies as an isometry exactly when for any distinct vertices $a$ and $b$ of $G$ the zigzag distance from $a$ to $b$ in $G$ equals the one from $f(a)$ to $f(b)$ in $G'$.
These absolute retracts enjoy other equivalent categorical properties also satisfied by graphs and ordered sets: \emph{injectivity}, \emph{$2$-Helly property} as well as the \emph{extension property}.

Despite the similarity with the two results mentioned above for graphs and ordered sets, there is an essential difference at the level of proofs. For the first two results, the set of distance values  can be endowed with a graph  structure  rendering  it an absolute retract such that each of our graphs can be isometrically embedded into a power of that graph. This approach cannot work for  oriented graphs. Indeed,  it is not true that in the  class of oriented graphs (with loops)  every graph embeds into an absolute retract.  For instance, an
oriented graph containing a directed cycle of length larger than $2$ cannot
be embedded in such an absolute retract.

In the proof we regard an oriented graph as a generalized metric space with respect to the zigzag distance. The basic ingredients of this approach are presented  in Sections 2 and 3. We refer to \cite{JaMiPo} and to \cite{kabil-pouzet2} for generalized metric spaces; the reader could compare with  \cite{espinola-khamsi} for  ordinary metric spaces.  Directed graphs appear in Section 4 as generalized metric spaces.  Then, in Section 5,  it is shown that the distance values in an absolute retract belong to the MacNeille completion of the ordered set of words over the alphabet ${+,-}$. The absolute retract can then be embedded isometrically into a product of zigzags. The main result is stated in Section 6.

\section{Dual integral involutive quantales}

Values of the zigzag distance on a directed graph are particular subsets of $\Lambda^{*}$   and, precisely,   final segments of $\Lambda^{*}$, once this set is equipped with the subword ordering. It turns out that several categorical properties of the class of directed graphs are mirrored in the set $\mathbf{F}(\Lambda^{*})$ of final segments of $\Lambda^{*}$,  the  reservoir for potential distances,  provided that this set  is
  equipped with the reverse order of inclusion, the concatenation of languages and an involution. This structure is then a  the dual of an integral involutive quantal,  and  directed graphs  can be viewed as metric spaces over it.

\subsection{The dual quantale $\mathbf F(\Lambda^{*})$}
Words over the alphabet  $\Lambda:= \{+,-\}$ are finite sequences of letters. We write a word $u$ with a
mere juxtaposition of its letters as $u:=
a_{0}\cdots a_{n-1}$ where $a_{i}$
are letters from $\Lambda$ for 0 $\leq i\leq n-1.$ The integer $n$ is the \textit{%
length} of the word $u$ and we denote it $\left| u \right| $.
Hence we identify letters with words of length $1$. We denote by $\Box $ the
\emph{empty word}, which is the unique  word of length zero.
The \emph{concatenation} of two word $u:= a_{0}\cdots a_{n-1}$ and $v:=b_{0}\cdots b_{m-1}$ is the word $uv:=a_{0}\cdots a_{n-1}b_{0}\cdots b_{m-1}$. We denote by $\Lambda^{*}$ the set of all words on the alphabet $\Lambda$. Once equipped with the
concatenation of words, $\Lambda^{*}$ is a monoid, whose neutral element is the empty word, in fact $\Lambda^{\ast}$ is the \textit{free
monoid} on $\Lambda$.
 On $\Lambda:= \{+,-\}$ one has the natural involution $-$ that interchanges the two letters. This involution can be lifted to $\Lambda^{*}$ by  reversing every word  and involuting its letters.
The monoid $\Lambda^{*}$ can be ordered by the \emph{subword ordering}: a word $u$ is a subword of $v$ if $u$ can be obtained from $v$ by deleting some letters). With this order, $\Lambda^{*}$ becomes an ordered monoid.

Involution and concatenation of words carry over to sets of words by word-wise application of the operations: If $X, Y \in \powerset (\Lambda^{\ast})$ the \emph{concatenation} of $X$ and $Y$ is the set  $XY:= \{uv: u\in X, v\in Y\}$.  This  operation  extends the concatenation operation on $\Lambda^{\ast}$. The involution  on $\Lambda^{\ast}$ extends  to $\powerset (\Lambda^{\ast})$ by setting $\overline X:= \{\overline {u}: u\in X\}$, it reverses the concatenation of languages and preserves the inclusion order on languages.  Final segments (or upper sets) are subsets $F$ of $\Lambda^{*}$ such that $u\in F$ and $u\leq v$ imply $v\in F$.
 The set $\mathbf{F}(\Lambda^{*})$ is a complete lattice with respect to inclusion. Then concatenation trivially distributes over  arbitrary unions on either side. The top set $\Lambda^{*}$ serves as a neutral element (unit) for concatenation on $\mathbf{F}(\Lambda^{*})$ because $\Lambda^{*}$  contains the empty word.

These properties thus describe $\mathbf{F}(\Lambda^*)$ - preliminarily endowed with inclusion as the ordering - as  an \emph{integral involutive quantale}. For the interpretation as distance values, however, we need to reverse the order and hence to put this structure upside down. Then we actually deal with the dual of an integral involutive quantale, so that intersection, concatenation, involution and $\Lambda^{*}$  as a zero become the fundamental operations.

In general, a  \emph{dual integral involutive quantale}  (${DI^2}$-\emph{quantale}, for short) is a complete lattice  $\mathcal D$ with  least element $0$ and  greatest
element $1$, equipped with a monoid operation $\oplus$ and an involution $^-$ such that:
The element 0 is the  neutral element of the operation;
the involution is order-preserving and reverses the monoid
operation, that is $$\overline {p\oplus q} = \bar {q} \oplus \bar {p}\;\;{\rm  holds}
\;\;{\rm for\; every\;}\; p, q \in \mathcal D, $$ and
the following distributivity condition holds for any index sets $A$ and $B$:
\[\bigwedge  _{\alpha \in A, \beta \in B} (p_\alpha \oplus q_\beta) =
\bigwedge _{\alpha \in A} p_\alpha  \oplus \bigwedge _{\beta \in B} q_\beta \]
for all $p_\alpha \in {\mathcal D}$ $(\alpha \in A)$ and
$p_{\beta} \in {\mathcal D}$ $(\beta \in B)$ or equivalently (because of the
involution)
\[\bigwedge _{\alpha \in A} (p_\alpha \oplus q)  =
\bigwedge _{\alpha \in A} p_\alpha\; \oplus\;  q \]
for all $p_\alpha \in {\mathcal D}$ $(\alpha \in A)$ and $q\in {\mathcal D}$. Note
that this distributivity condition entails the fact that the  monoid
operation and the ordering are compatible (that is  that is
$p\leq p^{'}$ and $q \leq q^{'}$ imply $p\oplus q\leq p^{'}\oplus q^{'}$).

We have borrowed this terminology from  Kaarli and Radeleczki  \cite{kaarli-radeleczki}. Elsewhere, $DI^2$-quantales were introduced under the name of Heyting algebra by the second author in \cite{pouzet}, see \cite{JaMiPo}, \cite{pouzet-rosenberg} or involutive Heyting algebras (\cite{kabil-pouzet}) in the context of generalized metric spaces.

The set   $\mathbf{F}(\Lambda^*)$ ordered by reverse of inclusion, with the concatenation of languages (that we do not denote by $\oplus$) and involution is a basic example of ${DI^2}$-{quantale}. There is an other that we introduce now.
\subsection{The MacNeille completion of the monoid of words}

We extend the  ordered monoid $\Lambda^*$  to a complete lattice-ordered monoid by applying the MacNeille completion,  which is in  some sense the least complete lattice extending $\Lambda^{\ast}$. The definition goes as follows. If $X$ is a subset of $\Lambda^{\ast}$ ordered by the subword ordering then

 $$X^{\Delta}:= \bigcap_{x \in X} \uparrow x\;  \text{and}\; X^{\nabla}:= \bigcap_{x \in X} \downarrow x$$
are respectively  the {\it upper cone} generated by $X$  and
 the {\it lower cone} generated by $X$, where $\uparrow  x$ and $\downarrow x$ denote the set of superwords and the set of subwords of $x$, respectively. The pair $(\Delta, \nabla)$ of maps on the complete
lattice of subsets of  $\Lambda^{\ast}$ constitutes a Galois connection. Thus,  a set $Y$ is an upper cone  if
and only if $Y = Y^{\nabla \Delta}$, while a set $W$ is
a lower cone if and only if $W = W^{\Delta \nabla}.$ This Galois connection
$(\Delta, \nabla)$ yields the {\it MacNeille completion} of
$\Lambda^{\ast}.$ This completion is realized  as the complete
lattice $\{W^{\nabla}:  W\subseteq \Lambda^{\ast}\}$
ordered by inclusion or $\{Y ^{\Delta} : Y\subseteq \Lambda^{\ast}\}$ ordered by reverse inclusion. In this paper,  we choose as completion the latter representation
denoted by $\mathbf {N}(\Lambda^{\ast})$. This complete lattice is studied in detail in \cite{bandelt-pouzet}.

We recall the important fact that the sets of the form  $W^{\nabla}$ for $W$ nonempty coincide with the nonempty finitely generated initial segments of $\Lambda^{\ast}$:

\begin{lem}\label{lem:jullien}Jullien \cite{jullien}The set $\mathbf {N}(\Lambda^{\ast})\setminus \{\emptyset\}$ is order isomorphic to the distributive lattice $\mathbf I_{<\omega}(\Lambda^{\ast})\setminus \{\emptyset\}$ ordered by inclusion  and comprising  the finitely generated initial segments of $\Lambda^{\ast}$. \end{lem}

 The concatenation, order and involution defined on $\mathbf {F}(\Lambda^{\ast})$ realize  it as a  $DI^{2}$-subquantale on the MacNeille completion  $\mathbf {N}(\Lambda^{\ast})$ of $\Lambda^*$(see Proposition 2.2 of \cite{bandelt-pouzet} and Section 3 of \cite{kabil-pouzet-rosenberg}). We recall the following syntactic characterization of  $\Lambda^*$.
\begin{proposition}\label{cancellation} \cite{bandelt-pouzet} Corollary 4.5.
A member  $Y$ of $\mathbf{F}(\Lambda^{*})$ belongs to $\mathbf {N}(\Lambda^{*})$ if and only if it satisfies the following cancellation rule:
if $u + v \in Y$ and $u - v \in Y$ then $u v \in Y$.
\end{proposition}

\section{The category of metric spaces over a dual integral involutive quantale}

With the zigzg distance, oriented  graphs and, more generally directed graphs yield     metric spaces over the $DI^2$-quantale $\mathbf {F}(\Lambda^{\ast})$ of final segments of $\Lambda^{*}$. In some cases, the metric is defined over the smaller $DI^2$-quantale $\mathbf {N}(\Lambda^{\ast})$ constituting the MacNeille completion of $\Lambda^{*}$.

We define a  \emph{metric},  alias \emph{distance},  $d$ on a set $V$  over a given $DI^2$-quantale $\mathcal D$ as a map
$d : V\times V \rightarrow {\mathcal D}$ satisfying the following properties
for all $x,y,z \in V$:

$\bullet$ $d(x,y)\;=\;0$ if and only if $x\;=\;y$,

$\bullet$  $d(x,y)\; \leq \;d(x,z)\;\oplus \;d(z,y)$,

$\bullet$ $d(x,y)\;=\;\overline{d(y,x)}$.

The pair $(V,d)$ is said to be a {\it (generalized) metric space} over $\mathcal D$. If the metric is fixed in a given context, then we simply speak of the metric space $V$ over $\mathcal D$.
In this space,  define the  \emph{ball}  with \emph{center} $x\in V$ and \emph{radius}
$r\in \mathcal D$ as
the set $B_V(x,r)=\{y \in V:d(x,y)\leq r\}$; if there is no danger of
confusion we will denote it $B(x,r)$ instead of $B_V(x,r)$.

If $(V,d)$ and $(V^{'},d^{'})$ are two metric spaces over $\mathcal D$,
then a map
$f : V \rightarrow V^{'}$ is {\it non-expansive} (also said sometimes {\it contracting})
provided that $d^{'}\left(f(x),f(y)\right)
\leq d(x,y)\;for\;all\;x,y\in V$.
If  equality holds for all $x,y\in V$, then $f$ is an
{\it isometry}. Hence in our terminology an isometry is not necessarily
surjective. We say that $V$ and $V^{'}$ are {\it isomorphic} if there is a surjective isometry from $V$ onto $V^{'}$.
If $V$ is a subset of $V^{'}$ and the identity map is non-expansive, we
say that $(V,d)$ is a {\it subspace} of $(V^{'},d^{'})$, or that
$(V^{'},d^{'})$ is an {\it extension } of $(V,d)$. If, moreover, this map is
an isometry (that is,  $d$ is the restriction of $d^{'}$ to $V \times V$),
then we call $(V,d)$ an {\it isometric subspace} of $(V^{'},d^{'})$ and
$(V^{'},d^{'})$ is an {\it isometric extension} of $(V,d)$.
As usual, $Hom(V,V^{'})$
denotes the set of all non-expansive maps from $V$ to $V^{'}$
(i.e. morphisms in the corresponding category).\\
\subsection{Absolute retracts in the category of metric spaces over $\mathcal D$}
We denote by $\mathcal M_{\mathcal D}$ the  category   of metric spaces  over $\mathcal D$,  with the non-expansive maps as morphisms. Due to the fact that joins exist in $\mathcal D$, this  category  has products. If $\left ((V_i,d_i)\right)_{i\in I}$ is a family of metric spaces over
$\mathcal D$, the direct product $(V,d) =
\displaystyle \prod_{i\in I}(V_i,d_i)$, is the cartesian product
$V= \displaystyle \prod_{i\in I} V_i$, equipped with the "sup"
(or $\ell^{\infty}$) distance $d : V\times V \to {\mathcal D}$ defined by:
\[d\Big((x_i)_{i\in I},(y_i)_{i\in I} \Big) =
\bigvee_{i \in I} d_i(x_i,y_i). \]

Let $V$ and $V'$ be two metric spaces over $\mathcal D$. The space $V$ is a
{\it retract} of $V'$  if there are non-expansive maps
$f: V \to V'$ and $g : V' \to V$ such that $g \circ f = id_V$ (where $id_V$
is the identity map on $V$). If this is the case $f$ is said to be a
{\it coretraction} and $g$ a {\it retraction}. If $V$ is a subspace of $V'$,
then $V$ is a retract of $V'$ if there is a non-expansive map from
$V'$ to $V$ such that $g(x) = x$ for all $x\in V$. We can easily see that
every coretraction is an isometry. A metric space is an
{\it absolute retract} if it is a retract of every isometric extension.

A metric space $V$ is said to be {\it injective} if for all spaces $V'$ and $V''$,
each non-expansive map $f : V' \to V$, and every isometry
$g : V' \to V''$ there is a non-expansive map $ h : V'' \to V$ such that
$ h\circ g =f$. The metric space $V$ has the \emph{extension property} if for every space $V'$, every non-expansive map  $f$ from a subset $A$ of $V'$ extends to a non-expansive map from $V'$ to $V$. With the help of Zorn's lemma, this amounts to the fact that every non-expansive map from a subset $A$ of $V'$ extends to every $x\in V'\setminus A$ to a non-expansive map from $A\cup\{x\}$ to $V$.

We say that a space $V$ is {\it hyperconvex} if the
intersection of every family
of balls $\left(B_V(x_i,r_i)\right)_{i\in I}$ is non-empty whenever
$d(x_i,x_j) \leq r_i \oplus \overline {r_j}$ for all $i, j \in I$.

\noindent Hyperconvexity is equivalent to the conjunction of the following conditions:\\
$(1)$ {\it Convexity} : for all $x,y \in V$ and $p,q \in \mathcal D$ such
that $d(x,y) \leq p \oplus q$ there is $z \in V$ such that $d(x,z)\leq p$
and $d(z,y)\leq q$.\\
$(2)$ {\it 2-Helly property}: the intersection of every set
(or, equivalently, every family) of balls is non-empty provided
that their pairwise intersections are all non-empty.

We  recall the definition of the distance over a $DI^2$-quantale $\mathcal D$.
It relies on the classical notion of {\it residuation}.
Let $v\in \mathcal D$. Given $\gamma \in {\mathcal D}$, the sets
 $\{r \in {\mathcal D}: v \leq r \oplus \gamma\}$ and
$\{r \in {\mathcal D}: v \leq  \gamma \oplus r \}$ have least elements, that we denote
respectively $\lceil v \oplus \ominus \gamma \rceil$ and $\lceil \ominus\gamma \oplus v  \rceil$
(where, in fact, $\overline {\lceil\ominus\gamma \oplus v \rceil} =
\lceil \bar v \oplus \ominus \bar \gamma \rceil$).

For
$p, q \in \mathcal D$, set
$$D(p,q):=\{r \in {\mathcal D}: p\leq q \oplus \bar r\;\;{\rm and}\; \; q\leq p\oplus r\}.$$ Due to the
distributivity condition on $\mathcal D$ this set has a least
element, that we denote $d_{\mathcal D}(p,q)$.

\begin{lem} \cite{JaMiPo}
For every metric space $(V,d)$ over $\mathcal D$, and for all $x, y\in V$,
the following equality holds:
$$d(x,y) =  \bigvee_{z\in V} d_{\mathcal D}\left(d(z,x),d(z,y)\right).$$
\end{lem}
From this lemma one infers the following fundamental observation.
\begin{proposition} (Proposition II.2-7 of \cite{JaMiPo})\label{metric-quantale}The map $(p,q) \mapsto d_{\mathcal D}(p,q)$ over a $DI^2$-quantale $\mathcal D$ is a distance
on $\mathcal D$, in fact $({\mathcal D},d_{\mathcal D})$ is an hyperconvex metric
space and every metric space over $\mathcal D$ embeds isometrically
into a power of  $({\mathcal D},d_{\mathcal D})$.
\end{proposition}

Hyperconvex spaces enjoy the extension property, hence they are injective. Since  hyperconvexity is preserved under the formation of products, every  metric space embeds isometrically into an injective object. This fact is shortly expressed by saying  that the category $\mathcal M_{\mathcal D}$ has \emph{enough injectives}. From that follows
 an  important structural property of the category of metric spaces over a $DI^2$-quantale.

\begin{proposition} (Theorem 1, section II-2.9 of \cite{JaMiPo}) \label{equivalenceAR} In the category of metric spaces over  a $DI^2$-quantale $\mathcal D$, injective, absolute retracts,  hyperconvex spaces, spaces with the extension property and retracts of power of $({\mathcal D},d_{\mathcal D})$ coincide.\end{proposition}

In the category of  metric spaces (over the non-negative reals), every metric space has an \emph{injective envelope} (also called an injective hull), a  major fact due to Isbell \cite{isbell}. One can view an injective envelope of a metric space $V$ as an hyperconvex isometric extension $\tilde V$ of $V$, which is minimal with respect to inclusion (that is, there is no proper hyperconvex subspace of $\tilde V$ containing isometrically $V$) and one can note that those minimal  extensions are isometric by the identity on  $V$. These facts extend to generalized metric spaces.
\begin{proposition}\label{prop:injectiveenvelope} (Theorem 2, section II-3.1 of \cite{JaMiPo})
Every metric space over a $DI^2$-quantale has an injective envelope.
\end{proposition}
In the category of ordered sets, the injective envelope of an ordered set is its MacNeille completion \cite{banaschewski-bruns}. As a consequence of Proposition \ref {prop:injectiveenvelope}, injective envelopes also exist in the category of undirected graphs as well as in the category  of directed graphs (see  \cite{JaMiPo} for details). We discuss briefly in Subsection \ref{injectivehull} the case of oriented graphs.
\section{Directed graphs as generalized metric spaces}

We  consider the $DI^{2}$-quantales  $\mathbf {F}(\Lambda^{*})$ and $\mathbf{N}(\Lambda^{*})$ and then the graphs for which the zigzag distance yield metric spaces over those quantales.

As said, a  (reflexive) graph is a \emph{zigzag} if its symmetric   hull is a path.
If $u:=a_0 \cdots a_i\cdots a_{n-1}$ is a word over the alphabet $\Lambda$,  the \emph{zigzag $Z_u$} with \emph{initial  vertex}  $\{0\}$ and \emph{final vertex}  $\{n\}$ is the reflexive oriented graph on $\{0, \dots n\}$ such that $(i,i+1)$ is an arc if  $a_i= +$ and, $(i+1,i)$ is an arc if  $a_i= -$.


If $G:=(V, \mathcal E)$ is a graph and $x, y$ are two vertices of $G$ then the zigzag distance $d_G(x,y)$ consists of words $u:= a_0  \cdots a_i \cdots  a_{n-1}$ over the alphabet $\Lambda^{\ast}$ for which  there is an homomorphism $f$ from the  zigzag $Z_u$    to $G$ sending the initial vertex $0$ to $x$ and the final vertex $n$ to $y$. If $x$ and $y$ are not connected in the symmetric hull of $G$, no such word exist,  hence $d_G(x,y)= \emptyset$.  It is a simple exercise to verify that the pair $(V, d_G)$ is a metric space over $\mathcal D:= \mathbf {F}(\Lambda^{*})$. Metric spaces over $\mathcal D:= \mathbf {F}(\Lambda^{*})$ such that the distance is the zigzag distance associated with a  reflexive directed graph are characterized by the following lemma.

\begin{lem}\label{metric-graph}
Let $(V,d)$ be a metric space over $\mathcal D:= \mathbf {F}(\Lambda^{*})$. The following properties are
equivalent:\\
$(i)$ The map $d$ is of the form $d_{G}$ for some  directed graph  $G:=(V,\mathcal E)$;\\
$(ii)$ For all $u$, $v \in \Lambda^{\ast}$ and  $x,y \in V$,
if $uv \in d(x,y)$, then there is
some $z\in V$ such that $u \in d(x,z)$ and $v \in d(z,y)$.
\end{lem}
\begin{proof}

$(ii) \Rightarrow (i)$. Let $G:=(V,\mathcal E)$ where
$\mathcal E:=\{ (x,y) : + \in d(x,y) \}$. Let $u:= a_0 \cdots a_{n-1} \in d_{G}(x,y)$. We prove that $u\in d(x,y)$. By definition of $d_G$, there is a homomorphism $f$ of the zigzag   $Z_u$ in $G$ sending its extremities $0$ and $n$ to $x$ and $y$ respectively. This means that  $(f(i),  f({i+1})) \in \mathcal E$ if $a_i= +$,  and  $(f({i+1}), f({i})) \in \mathcal E$ if  $a_i=-$ for each $i$.
By definition of $G$, $a_i \in d(f(i),  f({i+1})) $. From the triangular
inequality,  we get
$d(x,y) \supseteq d(f(0),f(1))\cdots d(f(n-1), f(n))$ (in multiplicative notation), hence
$u = a_0 \cdots a_{n-1} \in d(x,y)$.
Conversely, let $u: =a_0 \cdots  a_n \in d(x,y)$. From $(ii)$, there is a
sequence of elements $z_0: =x,\dots, z_n := y$  such that
$a_i \in d(z_i, z_{i+1})$ for all $i$,
$0 \le i < n$. Thus  $z_0,  \dots,  z_{n}$ is the image of the zigzag $Z_u$. It follows
$u \in d_{G}(x,y)$.

$(i) \Rightarrow (ii)$.  Let $u := a_0\cdots  a_{n-1}$,
$v:= b_0 \cdots b_{m-1}$, such that  $uv \in d(x,y)$. Since
$uv \in d(x,y)$ and $d= d_G$, there is a homomorphism $f$ from $Z_{uv}$  to $G$ sending $0$ to $x$ and $n+m$ to $y$. It suffices to take $z=f(n)$.
\end{proof}

Property $(ii)$ of Lemma \ref{metric-graph} is a weak form of convexity, hence:

\begin{cor} The distance of a metric spaces over $\mathcal D:=\mathbf{F}(\Lambda^{*})$ that satisfies  the convexity property, e.g. an  hyperconvex  space,
 is the zigzag distance associated with a reflexive directed graph.
\end{cor}
For $\mathcal D$ equal  to $\mathbf{F}(\Lambda^{*})$ as well as
$\mathbf{N}(\Lambda^{*})$, let $\mathcal G_{\mathcal D}$ be the class  of graphs whose the zigzag distance belongs to $\mathcal D$. With the homomorphisms of graphs, it becomes a  category.

According to Lemma \ref{metric-graph}, \emph{if $G$ and $G'$ are two reflexive directed graphs, a map $f$ from $V(G)$ to $V(G')$ is a graph homomorphism if and only if $f$ is non-expansive from the metric space associated to $G$ to the metric space associated to $G'$}. In other words:
\begin{proposition} As a category, $\mathcal G_{\mathcal D}$ identifies to a full subcategory of $\mathcal M_{\mathcal D}$, the  category   of metric spaces  over $\mathcal D$, with the non-expansive maps as morphisms.
\end{proposition}

We recall that the \emph{direct product of a family $(G_i)_{i\in I}$ of reflexive directed graphs} is the graph $\Pi_{i\in I} G_i$ whose vertex set is the direct product $\Pi_{i\in I} V(G_i)$, two vertices $(x_i)_{i\in I}$ and $(y_i)_{i \in I}$ forming a directed edge if and only if $(x_i, y_i)$ form a directed edge for each $i$. The reader will note that the existence of loops at  every vertex has a strong effect of the existence of edges on distinct vertices.

\begin{cor} \label{lem:product}
If $G$ is  the direct product of a family $(G_i)_{i\in I}$ of reflexive directed graphs then the zigzag distance $d_G$ is the sup-distance of the $d_{G_{i}}$'s.
\end{cor}

\begin{cor}A member $(V, d)$ of $\mathcal M_{\mathcal D}$ is an absolute retract iff the distance on $V$ comes from a graph and this graph is an absolute retract in $\mathcal G_{\mathcal D}$ with respect  to isometric embedding.  \end{cor}

To give a simple example of absolute retract,  take a directed path $P$ with two arcs, which passes from vertex $a$ through $b$ to vertex $c$. Assume that $P$ is isometrically embedded in some oriented graph $H$. How to construct a retraction from $H$ to $P$? Well, every vertex $x$ of $H$ is exactly one of the following types: (1) if some zigzag from $x$ to $a$ has exclusively forward arcs, then map $x$ to $a$; (2) if some zigzag from $x$ to $c$ has exclusively backward arcs, then map $x$ to $c$; (3) else, map $x$ to the intermediate vertex $b$. Note that types (1) and (2) exclude each other, for otherwise $P$ would lie on an oriented cycle through $x$ within $H$ and thus could not be isometric in $H$. For a similar reason, no vertex of type (3) can have a forward arc to a type (1) vertex or a backward arc to a type (2) vertex. This ensures that the map is a retraction onto $P$. Therefore $P$ is an absolute retract of  oriented graphs.

\begin{lem}\label{zigzag} The metric space associated to any   directed zigzag $Z$  has the extension property in the category $\mathcal M_{\mathcal D}$.  In particular, every non-expansive map sending two vertices of a reflexive directed graph $G$ on the extremities of $Z$ extends to a graph   homomorphism from $G$ to $Z$. \end{lem}
\begin{proof}
Let $Z$ be a directed zigzag (with loops). Its symmetric hull (obtained  by deleting the orientation of arcs in $Z$) is a path. The balls in $Z$ are intervals of that path, and each of these intervals  is either finite or the full path. Hence, every family of balls has the $2$-Helly property. Since convexity holds trivially, $Z$, as a metric space over $\mathbf{F}(\Lambda^*)$, is hyperconvex, hence according to   Theorem \ref{equivalenceAR}, it satisfies the extension property.
\end{proof}

\section{Isometric embedding into products of zigzags}
Being a $DI^{2}$-quantale, $\mathbf{N}({\Lambda^*})$ supports a  distance $d_{\mathbf {N}(\Lambda^*)}$ and this distance is the zigzag distance of a graph $G_{\mathbf{N} (\Lambda^*)}$. But,  it is not true that every oriented graph embeds isometrically into a power of that graph.

We will first exhibit some obvious obstructions to isometric embeddability into powers of zigzags.
\subsection{Obstructions and examples}

Note that the product of oriented zigzags is an acyclic directed graph, that is, it does not contain any directed cycle. Indeed, if there was some directed cycle $C$ of length $n$ greater than $2$  within such a product, then every projection must map $C$ onto a directed subpath of the corresponding zigzag factor. This, however, forces the image of $C$ to be a singleton.

This argument can be pushed a little bit further. For a product $Q$ of oriented zigzags, the directed cycles of length $n$ greater than $1$ are actually the only forbidden directed subgraphs. The $2$-cycle is forbidden as zigzags are oriented graphs. Since we already know that the product $Q$ is acyclic, it can be extended to a linearly ordered set  $\hat{Q}$  by  the Szpilrajn extension theorem. The power set of $V(Q)$ ordered by inclusion can be regarded as an isomorphic copy of the $V(Q)$-power  of the directed path of length $1$, into which  $\hat{Q}$ naturally embeds by associating the principal ideal with every element of $V(Q)$.


Quite another issue would be to determine the forbidden induced subgraphs for products of oriented zigzags.  Let $G$ be a subgraph of such a product and let $a$ and $b$ be vertices of $G$ joined by an arc from $a$ to $b$. For a directed path $P$ from $a$ to $b$ in $G$ of length larger than $2$,  the projection of $P$ onto any zigzag factor must equal the projection of $\{a,b\}$. For any arc $(x,y)$ of $P$ there is some zigzag factor where this arc is projected onto a non-loop. This enforces arcs $(a,y)$ and $(x,b)$ to exist in the product graph. By a trivial induction argument it follows that the subgraph induced by $V(P)$ in $G$ is transitive, that is, constitutes a linearly ordered set. This means that the products of oriented zigzags enjoy some conditional transitivity. The smallest graph that violates this condition is obtained from a directed $4$-cycle by reverting a single arc.

Thus for an undirected $4$-cycle there are only two possible orientations that lead to an oriented
graph embeddable into a product of zigzags. The first one has alternating arcs and the second one is
composed of a directed $2$-path and its reverse. For the former orientation, two zigzags suffice for an
embedding, whereas in the latter case three zigzags are needed. Further one can retract either graph to a
smaller one residing in the respective product and still including the corresponding $4$-cycle with its
orientation; see Fig. \ref{twoinjectivehull}. The vertices which were added (one or two, respectively) then "fill holes" in those $4$-vertex oriented graphs and are thus necessary for obtaining a retraction from the product. With the main
theorem of Section 6 the two graphs in Fig. \ref{twoinjectivehull} are the smallest absolute retracts containing the respective $4$-
vertex oriented graph (highlighted by the bold vertices).

\begin{figure}[h]
\centering
\includegraphics[width=1\textwidth]{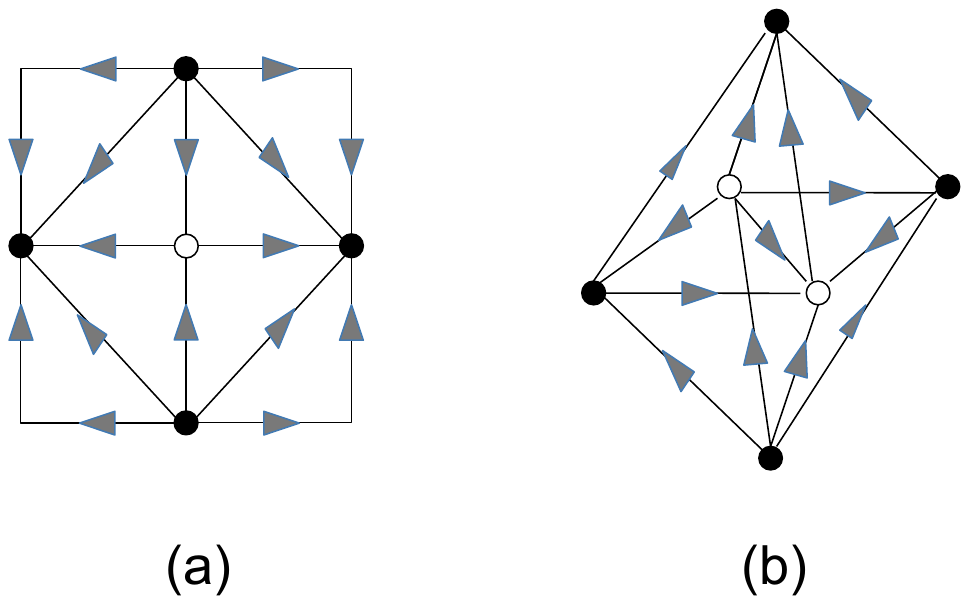}
\caption{}
\label{twoinjectivehull}
\end{figure}

\subsection{The embedding theorem}
\begin{theo}\label{theo:isometric}
	For a  directed graph   $G : = (V, \mathcal E)$ equipped with the zigzag distance,  the following properties are equivalent:
\begin{enumerate} [(i)]
	\item $G$ is isometrically embeddable into a product of  oriented zigzags;
	\item $G$ is isometrically embeddable into a power of $G_{\mathbf  {N}(\Lambda^*)}$;
	\item The values of the zigzag distance between  vertices of $V$ belong to $\mathbf{N}(\Lambda^*)$.
\end{enumerate}
\end{theo}

We may note that  the product can be infinite  even if the graph $G$ is finite. Indeed, if $G$ consists of two  vertices $x$ and $y$ with no value on the pair $\{x, y\}$ (that is the underlying graph is disconnected) then we need infinitely many zigzag of  arbitrarily long length.

\begin{proof}The proof follows the same lines as the proof of  Proposition IV-5.1 p.212 of \cite{JaMiPo}.\\
\noindent $(i)\Rightarrow (iii)$.
Let  $f$ be an isometric embedding of  $G$  into some product $\Pi_{i\in I}{G_i}$. From the definition of isometry and distance in the product we obtain the following equalities for every $x,y\in V$, 
\[(x_i)_{i\in I}:= f(x), (y_i)_{i\in I}:=f(y);\]
\[d_G(x, y) = d_{\Pi_{i\in I}{G_i}}(f(x), f(y))  = \bigvee\{d_{G_i}(x_i, y_i):  i  \in  I\}.\]
Since each $G_i$ is an oriented  zigzag,  $d_{G_i}(x_i, y_i) =\uparrow u_i$ where $u_i$ is the word associated with  the shortest zigzag $Z_{u_i}$ with initial and final vertices   $x_i$ and  $y_i$ in $G_i$. Since each set $\uparrow u_i$ is in $\mathbf{N}(\Lambda^*)$ the intersection of the $\uparrow u_i$'s, that is $d_G(x, y)$, is also  in $\mathbf{N}(\Lambda^*)$.

\noindent $(iii) \Leftrightarrow (ii)$. Apply Proposition \ref{metric-quantale}.

\noindent $(iii)\Rightarrow (i)$.  We use the following property.

\emph{For each pair  of vertices   $x,  y \in V(G)$  and each word  $u \in (d_G(x, y))^{\nabla}$, let  $Z_{u}$ be the zigzag
with end points $0$ and $n:=\ell(u)$ associated with $u$. Then, the map carrying $x$ onto $0$ and $y$ onto $n$ extends to a non-expansive map  $f_{x,y,u}$ from $G$ onto  $L_u$.}

The proof of this claim relies onto two facts. First, the distance zigzag from $\{0\}$ to $\{n\}$ in $Z_u$ is the final segment generated by the word $u$ which,  by our choice,   is smaller than $d_G(x,y)$. This  amounts to the fact that the partial map carrying $x$ onto $0$ and $y$ onto $n$ is a non-expansive map  from the subset $\{x,y\}$ of $G$ equipped with the zigzag distance into the space associated to the zigzag $Z_u$.
Now,  from Lemma \ref{zigzag},  this map extends to a non-expansive map from $G$ to  $Z_u$. \hfill $\Box$

 Let
$$G':= \Pi\{L_u: u\in (d_G(x, y))^{\nabla}\;  \text {and}\;  (x, y) \in V\times V \}.$$
The graph  $G$ is isometrically embeddable into  $G'$ by the map  $f$ defined by setting for every $z\in V$:
$$f(z):= \{f_{x,y,u}(z): u\in (d_G(x, y))^{\nabla}\;  \text {and}\;  (x, y) \in V\times V\}.$$   This map is  an isometry; indeed first, by definition of the product, it is non-expansive; next, to conclude that it is an isometry, it suffices to  check that for every $v \in \Lambda^*$, if  $d_G(x, y)\not \leq \uparrow v$ then  $d_{G'}(f(x), f(y))\not \leq \uparrow v$, that is for some triple $i:= (x',y',u)$ one has $d_{G_i}(f_i(x), f_i(y))\not \leq \uparrow u$. Let $v$ and  $x, y$ such that $d_G(x, y)\not \leq  \uparrow v $, (this amounts to  $v \not \in d_G(x, y)$). Since  $d_G(x, y)  = ((d_G(x, y))^{\nabla})^{\Delta}$ there is some   $u\in (d_G(x, y))^{\nabla}$  such that
$u \not \leq  v$. We may set $i:=(x, y, u)$.  \end{proof}
\subsection{Embeddings of absolute retracts}
\begin{theo}\label{thm:ARcompletion}
If an oriented graph $G:= (V, \mathcal E)$ is an absolute retract in the category of  oriented graphs then the values $d_G(x,y)$ of the zigzag distance $d_G$ belong to $\mathbf{N}(\Lambda^*)$, the MacNeille completion of $\Lambda^*$.
\end{theo}
\begin{proof} The proof has two major steps.
 First, we prove that $G$ has no $3$-element cycle. Second, we prove that the zigag distance between two vertices of $G$ satisfies the cancellation rule. From Proposition \ref{cancellation} it belongs to $\mathbf{N}(\Lambda^*)$.

$\bullet$ \emph{$G$ has no $3$-element cycle.}

Consider the directed cycle $C$ with three vertices and three arcs. Suppose by way of contradiction that $C$ sits inside  $G$. Then add disjoint sets $Y$ and $Z$ of equal cardinality which exceed the cardinality of $G$. Select a perfect matching between $Y$ and $Z$ and orient its edges from $Z$ to $Y$. All unmatched pairs $(y,z)$ between $Y$ and $Z$ are linked with an arc from $y$ to $z$.  Further add all forward arcs from the vertices of $C$  to the vertices of $Y$ and add arcs in the other direction from $Z$ to $C$. Add loops to the new vertices. It is easy to check that  $C$ and hence $G$ are isometrically embedded in this extension $H$ of $G$ (see Figure \ref{decoration-cycle}). By cardinality reasons,  two distinct vertices $y$ and $y'$ of $Y$ must receive the same image under a retraction $f$ from $H$ to $G$, which would be guaranteed if $G$ was an absolute retract. Then the matched companion $z \in Z$ of $y$ must be mapped onto the same vertex $v$ in $G$ as $y$ because otherwise the image would receive a directed $2$-cycle. Now, however, $v$ has a forward and backward arc to each vertex of $C$, which is impossible. We conclude that $G$ cannot include a directed $3$-cycle.\hfill $\Box$

\begin{figure}[h]
\centering

\includegraphics[width=0.6\textwidth]{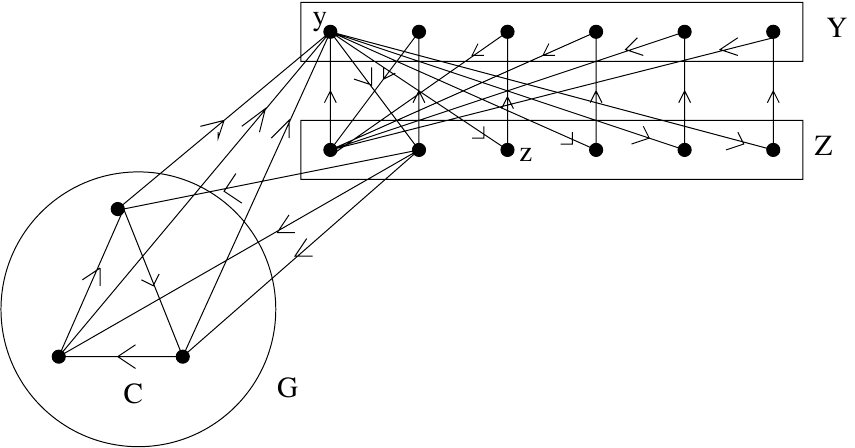}
           \caption{An isometric extension of an oriented graph containing a $3$-cycle}
              \label{decoration-cycle}
\end{figure}

$\bullet$ \emph{The zigzag distance $d_G(x,y)$ from a  vertices $x$ to a vertex $y$ satisfies the cancellation rule.}

Suppose that  $d_G(x,y)$  contains  two words  $u+v$  and  $u-v$. We claim  that it contains $uv$.  If $u$ and $v$ are the empty word then, since $G$ is oriented, $x=y$ hence $\Box \in d_G(x,y)$ and our claim holds. Without loss of generality,  we may suppose that $v\not= \Box$, hence $v= \alpha v'$ where $\alpha \in \{+,-\}$. Add to $G$ an oriented  zigzag $L_u$,  corresponding to the word $u$,   from the vertex $x$ to a new vertex $x'$,  and add an oriented  zigzag $L_{v'}$, corresponding to the word $v'$  from a vertex $y'$ to the vertex $y$. Then  glue  $x'$ to a  $3$-cycle  $\{x', z',z''\}$ and link $z',z''$ to $y'$ by two  edges, forward if $\alpha=+$, backward if $\alpha=-$  (see Figure \ref{figure-decoration2}).
The resulting graph $G'$ is oriented and, as this can be easily checked,  is an isometric extension of $G$. Since $G$ is an absolute retract, there is a retraction $f$ from $G'$ onto $G$. Since $G$ contains no $3$-element cycle, the images of the $3$-cycle $x',z',z''$ collapse to a single vertex $z$. Since $f$ is nonexpansive, $u\in d_G(x,z)$ and $v=\alpha v'\in d_G(z,y)$, hence $uv\in d_G(x,y)$ as claimed. \hfill $\Box$

\end{proof}

\begin{figure}[h]
\centering
\includegraphics[width=0.4\textwidth]{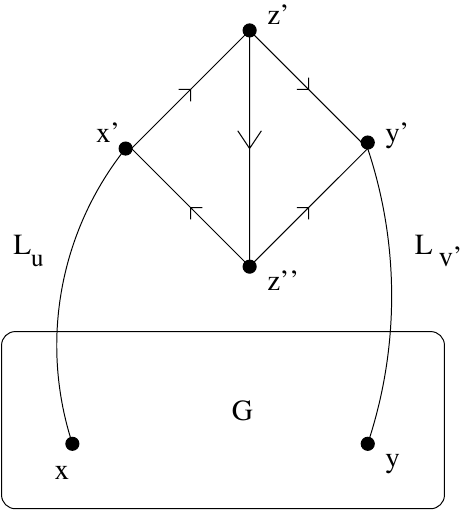}
\caption{An isometric extension with a $3$-cycle}
\label{figure-decoration2}
\end{figure}

\subsection{Injective hull}\label{injectivehull}
Any minimal absolute retract of an oriented graph with MacNeille-closed distances can also be referred to as the hyperconvex or injective hull. It would be of interest to describe all the hulls of (embeddable) small oriented graphs, say, up to $5$ vertices, thus mimicking an illustrative approach by Dress \cite{dress} in the case of ordinary metric spaces. Here is an example with $4$ vertices (see Figure \ref{twoinjectivehull}): the acyclic graph obtained by gluing together two directed paths at their corresponding end vertices. The injective hull is obtained as a retract of the product of a directed 2-path and two directed 1-paths. The two vertices added for the hull are indispensable because either of them fills a "hole" in the $4$-vertex graph. Their connection by an arc is then the only way to arrive at an absolute retract of oriented graphs.


\section{The main result}
Here,  we characterize absolute retracts among antisymmetric graphs.

\begin{theo}\label{theo-antisym}
	Let   $G: = (V, \mathcal E)$ be a   graph and $\Lambda: = \{+, -\}$. The following properties are equivalent:
	\begin{enumerate}[{(i)}]
	\item $G$ is an absolute retract in the category of  oriented graphs, with isometries as an approximation of coretractions;
	\item $G$ is injective in the category of  oriented graphs, with isometries as an approximation of coretractions
	\item $G$ is  an oriented graph and has the extension property among oriented graphs;
	\item G is an oriented graph  and the family of balls   $B(x, \uparrow r)$ with center $x\in V$ and radius $ \uparrow r \in  \mathbf {N}(\Lambda^*)$ with $r\in \Lambda^*$ has the $2$-Helly property;
	\item $G$ is a retract of a product of oriented  zigzags;
	\item  $G$ is a retract of a power of the oriented graph $G_{\mathbf{N}(\Lambda^*)}$ defined on $\mathbf{N}(\Lambda^*)$.
	\end{enumerate}
	\end{theo}
	Theorem \ref{theo-antisym} rectifies Th\'eor\`eme IV-3.1 stated in Jawhari, Misane and Pouzet \cite{JaMiPo}.

\begin{proof}
We prove the sequence of implications
$(vi)\Rightarrow (v) \Rightarrow (iv)\Rightarrow  (iii)\Rightarrow (ii)\Rightarrow (i)\Rightarrow (vi)$.
The substantial part of the proof is the implication $(i)\Rightarrow (vi)$.
We start with it.
Let $G$ be an absolute retract in the category of  oriented graphs. According to Theorem \ref{thm:ARcompletion}, the values $d_G(x,y)$ of the zigzag distance $d_G$ belong to $\mathbf{N}(\Lambda^*)$, the MacNeille completion of $\Lambda^*$. According to Proposition \ref{metric-quantale}, the space $(G, d_G)$ embeds isometrically into a power of the space $(\mathbf {N}(\Lambda^*), d_{\mathbf{N}(\Lambda^*)})$. This means that w.r.t. the zigzag distance, the graph $G$ isometrically embeds into the graph associated with this power of $(\mathbf{N}(\Lambda^*), d_{\mathbf{N}(\Lambda^*)})$. According  to Corollary  \ref{lem:product},  this graph is a  power of the graph $G_{\mathbf{N}(\Lambda^*)}$ associated to $(\mathbf{N}(\Lambda^*), d_{\mathbf{N}(\Lambda^*)})$. Since $G$ isometrically embeds into a power of $G_{\mathbf{N}(\Lambda^*)}$ and $G$ is an absolute retract, it is a retract of this power, hence $(vi)$ holds.

$(vi)\Rightarrow (v)$. First, $G$ is an absolute retract in the category of graphs whose values of the zigzag distance belongs to  $\mathbf {N}(\Lambda^{\ast})$. Indeed, since $\mathbf {N}(\Lambda^{\ast})$ is a $DI^{2}$-quantale, absolute retracts in the category of metric spaces over $\mathbf {N}(\Lambda^*)$ are retracts of powers of  the metric space  $(\mathbf {N}(\Lambda^*), d_{\mathbf{N}(\Lambda^*)})$ (Proposition  \ref{equivalenceAR}).  Hence,  absolute retracts among graphs whose zigzag distance belongs to  $\mathbf {N}(\Lambda^{\ast})$ are retracts of power  of $G_{\mathbf{N}(\Lambda^*)}$.  Since $G$ is a retract of a power of $G_{\mathbf{N}(\Lambda^*)}$, it is an absolute retract as claimed and furthermore   it embeds isometrically in that power.  By implication $(ii)\Rightarrow (i)$ of Theorem  \ref{theo:isometric}, it embeds isometrically into a product of  oriented zigzags. Since the metric spaces associated to these  zigzags belong to the category of metric spaces over $\mathbf{N}(\Lambda^*)$ and $G$ is an absolute retract in this category,  $G$ is a retract of a product of oriented zigzags, hence $(v)$ holds.

 $(v)\Rightarrow (iv)$.  For a reflexive oriented zigzag $Z$, the family of ball   $B(x, \uparrow r)$ with center $x\in V(Z)$ and radius  $\uparrow r$ with  $r\in  \Lambda^*$ has the $2$-Helly property (Lemma \ref{zigzag}).  This property being
preserved by products and retracts, it holds for $G$. Hence, $(iv)$ holds.

 $(iv)\Rightarrow  (iii)$. Suppose that $(iv)$ holds. Let $G'$ be a reflexive oriented graph, let $A$ be a subset of $V(G')$,  $f$ be a non-expansive map w.r.t. the zigzag distance from $A$ to $G$ and $x\in V(G')\setminus A$. The map  $f$ extends to $A\cup \{x\}$. Indeed, let $\mathcal B:= \{B(f(y),  \uparrow r): r \in d_{G'}(y, x), y\in A\}$. This family of balls has the $2$-Helly property. Hence, its intersection is non-empty. Pick $z$ in this intersection and set $f(x):= z$. This is an extension of $f$ to $A\cup \{x\}$.
 Hence $(iii)$ holds.

$(iii) \Rightarrow (ii)$. Let $G'$ be an oriented graph, $H$ be an isometric subgraph  of $G'$,  $f$ be a non-expansive map w.r.t. the zigzag distance from $H$ to $G$. Since $G$ has the extension property, we may extend $f$, a vertex after an other, to a non-expansive map from $G'$ to $G$. This proves that $G$ is injective. Hence $(ii)$ holds.

 $(ii)\Rightarrow (i)$. Let $G$ be an  injective  oriented graph. Let $G'$ be an oriented  isometric extension of $G$ w.r.t. the zigzag distance. The identity map on $V(G)$ is  non-expansive; since $G$ is injective, this map extends to $G'$ to a non-expansive map onto $G$. Hence $G$ is a retract and $(i)$ holds.
\end{proof}

If one considers directed graphs, not necessarily oriented, there is  a similar characterization.   If we drop Item $(v)$ and if in  $(vi)$  we replace the graph $G_{\mathbf{N}(\Lambda^*)}$   by the graph  $G_{\mathbf{F}(\Lambda^*)}$ defined on the collection $\mathbf {F}(\Lambda^*)$ of final segments of $\Lambda^{*}$, the above equivalences hold  (see Theorem 3 of \cite{kabil-pouzet}). We may extend these equivalences  to an analog of Item $(v)$, but it is  not sufficient to replace  oriented zigzags by  directed zigzags. Zigzags have to be replaced   by graphs associated to injective envelopes of $2$-element metric spaces whose distance between end-points is the complement in $\Lambda^{*}$ of a principal initial segment of  $\Lambda^{*}$ (see Theorem 4 of \cite{kabil-pouzet}).

A consequence of Theorem \ref{theo-antisym} is the following.

\begin{cor}
In the  following five categories,  the   absolute retracts are the same:  the categories of directed graphs with distance values in  $\mathbf {N}(\Lambda^{*})$,  of oriented graphs, of metric spaces over $\mathbf {N}(\Lambda^{*})$, of directed graphs and of metric spaces  over ${\mathbf {F}(\Lambda^{*})}$.
\end{cor}

Indeed, if $G$ is an  absolute retract  in the category  of oriented graphs, then it follows from Theorem \ref{theo-antisym}  that $G$ is an absolute retract into the category of directed graphs with distance values in  $\mathbf {N}(\Lambda^{*})$. Since the category of metric spaces over $\mathbf {N}(\Lambda^{*})$ has enough injectives, and injectives are metric spaces associated with graphs, then $G$, or rather the metric space associated to $G$,  is an absolute retract in this category. Since by $(iv)$ of Theorem  \ref{theo-antisym} the family of balls   $B(x, \uparrow r )$ with center $x\in V$ and radius  $\uparrow  r $ with $r\in  \Lambda^*$ has the $2$-Helly property  and this property is a characteristic property of injectivity in the categories of directed graphs and metric spaces  over ${\mathbf {F}(\Lambda^{*})}$ (apply Theorem 3 of \cite{kabil-pouzet}), the graph   $G$ is an absolute retract in these categories.

\section{Concluding remarks}
The category of oriented graphs is a full subcategory
 of the category of directed graphs. Several features of absolute retracts are thus shared. The main difference is that the former category has not enough injectives and the latter has much more complicated indecomposable absolute retracts which serve as factors in a product representation \cite{kabil-pouzet}. The general approach of \cite{JaMiPo} worked well for oriented graphs except that the modification on $\mathbf{N}(\Lambda^{*})$ by deleting one particular upper set of words does not yield a $DI^{2}$-quantale. The necessary correction as well as other results along a similar vein, notably for multigraphs,  were  obtained and exposed in Chapter V of  \cite{Sa}.

 A reader may wonder why to study absolute retracts of oriented graphs. A reason, also valid for the study of absolute retracts of generalized metric  spaces, is the  fixed point property. A metric space has the \emph{fixed point property} (f.p.p. for short) if every non-expansive map has a fixed point. An elementary fact is that retracts of metric spaces with the f.p.p. inherit it. Thus, if there are many spaces with the f.p.p., absolute retracts are good candidates. The fixed-point theorems of Tarski for complete lattices  and of Sine and Soardi  for bounded hyperconvex ordinary spaces illustrate this point (see \cite{espinola-khamsi}). A graph  has the \emph{fixed point property} if every edge-preserving map  has a fixed point. Clearly, such a graph must be oriented. In a recent paper  \cite{khamsi-pouzet}, the authors give an other illustration, proving that on retracts of powers  of  a finite oriented zigzag every commuting set of edge-preserving maps has a common fixed point. If the oriented zigzag  has only two vertices, this is an equivalent statement to the Tarski fixed-point theorem for commuting set of order preserving maps on a  complete lattice.  It remains to know if the transitive closure of a retract of a product of oriented zigzags  is a retract of a product of fences.

\section*{Acknowledgement} The first author thanks  the  Institut Camille Jordan, CNRS UMR 5208 for supporting his stay in Lyon in February 2019. The second author thanks    the members of the Maths Department of the University of Hamburg for their hospitality during his stay in February 2016.

\end{document}